\documentclass[12pt,oneside]{amsart}

\usepackage{amssymb}
\usepackage{amscd}

\title[A birational embedding with two Galois points]{A birational embedding of an algebraic curve into a projective plane with two Galois points} 
\author{Satoru Fukasawa}

\address{Department of Mathematical Sciences, Faculty of Science, Yamagata University, Kojirakawa-machi 1-4-12, Yamagata 990-8560, Japan} 
\email{s.fukasawa@sci.kj.yamagata-u.ac.jp} 
\thanks{The author was partially supported by JSPS KAKENHI Grant Number 16K05088.}

\newtheorem{theorem}{Theorem}
\newtheorem{proposition}{Proposition}

\newtheorem{remark}{Remark}

\begin{document}
\begin{abstract} 
A criterion for the existence of a birational embedding of an algebraic curve into a projective plane with two Galois points is presented. 
Several novel examples of plane curves with two inner Galois points as an application are described. 
\end{abstract}

\maketitle 

\section{Introduction} 
Let $C$ be a (reduced, irreducible) smooth projective curve over an algebraically closed field $k$ of characteristic $p \ge 0$ with $k(C)$ as its function field. 
We consider a rational map $\varphi$ from $C$ to $\Bbb P^2$, which is birational onto its image. 
For a point $P \in \Bbb P^2$, if the function field extension $k(\varphi(C))/\pi_P^*k(\Bbb P^1)$ induced by the projection $\pi_P$ is Galois, then $P$ is called a Galois point for $\varphi(C)$. 
This notion was introduced by Yoshihara (\cite{miura-yoshihara, yoshihara1}).
Furthermore, if a Galois point $P$ is a smooth point of $\varphi(C)$ (resp. is contained in $\Bbb P^2 \setminus \varphi(C)$), then $P$ is said to be inner (resp. outer). 
The associated Galois group at $P$ is denoted by $G_P$. 
If $\varphi: C \rightarrow \varphi(C)$ is isomorphic, then the number of Galois points is completely determined (\cite{fukasawa2, yoshihara1}). 
However, determining the number in general is difficult. 
For example, only seven types  of plane curves with two inner Galois points are known (see the Table in \cite{yoshihara-fukasawa}). 
It is important to find a condition for the existence of two Galois points. 

The following proposition is presented after discussions with Takahashi \cite{takahashi}, Terasoma \cite{terasoma} and Yoshihara \cite{yoshihara3}. 

\begin{proposition} \label{ftt}
Let $C$ be a smooth projective curve. 
Assume that there exist two finite subgroups, $G_1$ and $G_2$, of the full automorphism group ${\rm Aut}(C)$ such that $G_1 \cap G_2=\{1\}$ and $C/{G_i} \cong \Bbb P^1$ for $i=1, 2$. 
Let $f$ and $g$ be generators of function fields of $C/{G_1}$ and $C/{G_2}$, respectively. 
Then, the rational map 
$$ \varphi: C \dashrightarrow \Bbb P^2; \ (f:g:1)$$
is birational onto its image, and two points $P_1=(0:1:0)$ and $P_2=(1:0:0)$ are Galois points for $\varphi(C)$. 
\end{proposition}

For both points $P_1$ and $P_2$ to be inner, or outer, we need additional conditions. 
In this article, we present the following criterion. 

\begin{theorem} \label{main}
Let $C$ be a smooth projective curve and let $G_1$ and $G_2$ be different finite subgroups of ${\rm Aut}(C)$. 
Then, there exist a morphism $\varphi: C \rightarrow \Bbb P^2$ and different inner Galois points $\varphi(P_1)$ and $\varphi(P_2) \in \varphi(C)$ such that $\varphi$ is birational onto its image and $G_{\varphi(P_i)}=G_i$ for $i=1, 2$, if and only if the following conditions are satisfied. 
\begin{itemize}
\item[(a)] $C/{G_1} \cong \Bbb P^1$ and $C/{G_2} \cong \Bbb P^1$. 
\item[(b)] $G_1 \cap G_2=\{1\}$. 
\item[(c)] There exist two different points $P_1$ and $P_2 \in C$ such that 
$$P_1+\sum_{\sigma \in G_1} \sigma (P_2)=P_2+\sum_{\tau \in G_2} \tau (P_1) $$
as divisors. 
\end{itemize} 
\end{theorem} 

\begin{remark} 
\begin{itemize}
\item[(1)] For outer Galois points, we have to only replace (c) by 
\begin{itemize}
\item[(c')] There exists a point $Q \in C$ such that $\sum_{\sigma \in G_1} \sigma (Q)=\sum_{\tau \in G_2} \tau(Q)$ as divisors. 
\end{itemize}
\item[(2)] Yoshihara \cite{yoshihara3} gave a different criterion from ours, for two outer Galois points. 
Conditions (a) and (b) appear, but condition (c') does not appear in the statement of Yoshihara's criterion. 
\end{itemize} 
\end{remark}

We present the following application for rational or elliptic curves.  

\begin{theorem} \label{main_rational}
Let $p \ne 2$. 
Then, there exist the following morphisms $\varphi: \Bbb P^1 \rightarrow \Bbb P^2$, which are birational onto their images.  
\begin{itemize}
\item[(1)] $\deg \varphi(C)=5$ and there exist two Galois points $\varphi(P_1)$ and $\varphi(P_2) \in \varphi(C)$ such that $G_{\varphi(P_i)} \cong \Bbb Z/4\Bbb Z$ for $i=1, 2$, in $p \ne 3$. 
\item[(2)] $\deg \varphi(C)=5$ and there exist two Galois points $\varphi(P_1)$ and $\varphi(P_2) \in \varphi(C)$ such that $G_{\varphi(P_i)} \cong (\Bbb Z/2\Bbb Z)^{\oplus 2}$ for $i=1, 2$. 
\item[(3)] $\deg \varphi(C)=5$ and there exist two Galois points $\varphi(P_1)$ and $\varphi(P_2) \in \varphi(C)$ such that $G_{\varphi(P_1)} \cong \Bbb Z/4\Bbb Z$ and $G_{\varphi(P_2)} \cong (\Bbb Z/2\Bbb Z)^{\oplus 2}$. 
\item[(4)] $\deg \varphi(C)=6$ and there exist two Galois points $\varphi(P_1)$ and $\varphi(P_2) \in \varphi(C)$ such that $G_{\varphi(P_i)} \cong \Bbb Z/5\Bbb Z$ for $i=1, 2$. 
\end{itemize}
\end{theorem} 

\begin{theorem} \label{main_elliptic}
Let $p \ne 3$ and let $E \subset \Bbb P^2$ be the curve defined by $X^3+Y^3+Z^3=0$. 
Then, there exists a morphism $\varphi: E \rightarrow \Bbb P^2$ such that $\varphi$ is birational onto its image, $\deg \varphi(E)=4$, and there exist two inner Galois points for $\varphi(E)$. 
\end{theorem}
\section{Proof of the main theorem} 

\begin{proof}[Proof of Proposition \ref{ftt}] 
Let $P_1=(0:1:0)$ and $P_2=(1:0:0)$. 
Then, the projection $\pi_{P_1}$ (resp. $\pi_{P_2}$) is given by $(x:y:1) \mapsto (x:1)$ (resp. $(x:y:1) \mapsto (y:1)$), and hence, $\pi_{P_1} \circ \varphi=(f:1)$ (resp. $\pi_{P_2} \circ \varphi=(g:1)$). 
We have to only show that $k(C)=k(f, g)$. 
Since $k(C)/k(f)$ is Galois, there exists a subgroup $H_1$ of $G_1$ such that $H_1={\rm Gal}(k(C)/k(f, g))$. 
Similarly, there exists a subgroup $H_2$ of $G_2$ such that $H_2={\rm Gal}(k(C)/k(f, g))$.  
Since $G_1 \cap G_2=\{1\}$, $H_1=H_2=\{1\}$. 
Therefore, $k(C)=k(f, g)$. 
\end{proof}

\begin{proof}[Proof of Theorem \ref{main}] 
We consider the only-if part. 
Let $\varphi(P_1)$ and $\varphi(P_2) \in \varphi(C)$ be inner Galois points such that $G_{\varphi(P_i)}=G_i$ for $i=1, 2$. 
Assertion (a) is obvious. 
The proof of assertion (b) is similar to \cite[Lemma 7]{fukasawa1}. 
Let $D$ be the divisor induced by the intersection of $C$ and the line $\overline{P_1P_2}$, where $\overline{P_1P_2}$ is the line passing through $P_1$ and $P_2$. 
We can consider the line $\overline{P_1P_2}$ as a point in the images of $\pi_{P_1} \circ \varphi$ and $\pi_{P_2} \circ \varphi$. 
Since $\pi_{P_1} \circ \varphi$ (resp. $\pi_2 \circ \varphi$) is a Galois covering and $P_2 \in C \cap \overline{P_1P_2}$ (resp. $P_1 \in C \cap \overline{P_1P_2}$), 
$$(\pi_{P_1} \circ \varphi)^*(\overline{P_1P_2})=\sum_{\sigma \in G_1}\sigma(P_2) \ \  \left(\mbox{resp. } (\pi_{P_2} \circ \varphi)^*(\overline{P_1P_2})=\sum_{\tau \in G_2}\tau(P_1)\right)$$ as divisors (see, for example, \cite[III.7.1, III.7.2, III.8.2]{stichtenoth}). 
On the other hand, it follows that $(\pi_{P_1} \circ \varphi)^*(\overline{P_1P_2})=D-P_1$ (resp. $(\pi_{P_2} \circ \varphi)^*(\overline{P_1P_2})=D-P_2$).  
Therefore, 
$$ D=P_1+\sum_{\sigma \in G_1}\sigma(P_2)=P_2+\sum_{\tau \in G_2}\tau(P_1), $$
which is nothing but assertion (c). 

We then consider the if-part. 
Let $D$ be the divisor 
$$ D=P_1+\sum_{\sigma \in G_1}\sigma(P_2)=P_2+\sum_{\tau \in G_2}\tau(P_1), $$
by $(c)$. 
Let $f$ and $g \in k(C)$ be generators of $k(C/{G_1})$ and $k(C/{G_2})$ such that $(f)_{\infty}=D-P_1$ (i.e. $f \in \mathcal{L}(D-{P_1})$) and $(g)_{\infty}=D-P_2$, by (a), where $(f)_{\infty}$ is the pole divisor of $f$.   
Then, $f, g \in \mathcal{L}(D)$. 
Let $\varphi: C \rightarrow \Bbb P^2$ be given by $(f:g:1)$. 
Similar to Proposition \ref{ftt}, by (b), $\varphi$ is birational onto its image. 
The sublinear system of $|D|$ corresponding to $\langle f, g, 1\rangle$ is base-point-free, since ${\rm supp}(D) \cap {\rm supp}((f)+D)=\{P_1\}$ and ${\rm supp}(D) \cap {\rm supp}((g)+D)=\{P_2\}$. 
Therefore, $\deg \varphi(C)=\deg D$, and the morphism $(f:1)$ (resp. $(g:1)$) coincides with the projection from the smooth point $\varphi(P_1) \in \varphi(C)$ (resp. $\varphi(P_2) \in \varphi(C)$). 
\end{proof} 

\section{Applications} 

First, we consider rational curves. 

\begin{proof}[Proof of Theorem \ref{main_rational}] 
(1). 
Let $\sigma, \tau \in {\rm Aut}(\Bbb P^1)$ be represented by 
$$ \left(\begin{array}{cc} 
1 & -1 \\
1 & 1 
\end{array} \right), \ 
\left(\begin{array}{cc} 
0 & 1 \\
-\frac{1}{2} & 1 
\end{array} \right) $$
respectively, by assuming $p \ne 2$. 
Let $G_1=\langle \sigma \rangle$, $G_2=\langle \tau \rangle$, $P_1=(2:1)$ and $P_2=(-1:1)$. 
If $p \ne 3$, then $P_1 \ne P_2$ and condition (a) in Theorem \ref{main} is obviously satisfied. 
Note that
$$\{\sigma^i(P_2)| i=1, 2, 3\}=\{(1:0), (1:1), (0:1)\}=\{\tau^i(P_1)|i=1, 2, 3\}. $$
Condition (c) in Theorem \ref{main} is satisfied.    
Furthermore, $\sigma^4=1$ and $\tau^4=1$. 
We prove condition (b) in Theorem \ref{main}. 
Assume by contradiction that $\sigma^i=\tau^j$ for some $i, j$. 
If $i=1$ or $3$, then there exists an integer $l$ such that $(\sigma^i)^l=\sigma$. 
Then, $\tau^{jl}(0:1)=\sigma (0:1)=(-1:1)$. 
However, there exists no integer $i$ such that $\tau^i(0:1)=(-1:1)$. 
This is a contradiction. 
Therefore, $i=2$ and $j=2$. 
However, $\sigma^2(1:0)=(0:1) \ne (1:1)=\tau^2(1:0)$. 

(2). 
Let $\alpha \ne 0, 1, -1$ and let $\sigma_{\alpha}, \tau_{\alpha} \in {\rm Aut}(\Bbb P^1)$ be represented by 
$$ \left(\begin{array}{cc} 
0 & 1 \\
\alpha & 0
\end{array} \right), \ 
\left(\begin{array}{cc} 
1 & -\frac{1}{\alpha} \\
1 & -1 
\end{array} \right), $$
respectively. 
Let $G_\alpha=\langle \sigma_{\alpha}, \tau_{\alpha}\rangle$. 
Since $\sigma_{\alpha}\tau_{\alpha}=\tau_{\alpha}\sigma_{\alpha}$, $G_{\alpha} \cong (\Bbb Z/2\Bbb Z)^{\oplus 2}$. 
We take $\alpha' \ne 0, 1, -1, \alpha$. 
Let $G_1=G_{\alpha}$, $G_2=G_{\alpha'}$, $P_1=(1:\alpha')$ and $P_2=(1:\alpha)$. 
Note that
$$ \{\sigma_{\alpha}(P_2), \tau_{\alpha}(P_2), \sigma_{\alpha}\tau_{\alpha}(P_2)\}=\{(1:1), (0:1), (1:0)\}=\{\sigma_{\alpha'}(P_1), \tau_{\alpha'}(P_1), \sigma_{\alpha'}\tau_{\alpha'}(P_1)\}. $$
Condition (c) in Theorem \ref{main} is satisfied. 
Conditions (a) and (b) are obviously satisfied in this case. 

(3). 
We take $\sigma$ as in (1) and $\sigma_{\alpha}, \tau_{\alpha}$ as in (2). 
Let $P_1=(1:\alpha)$ and $P_2=(1: -1)$. 
Note that  
$$ \{\sigma^i(P_2)|i=1, 2, 3\}=\{(1:0), (1:1), (0:1)\}=\{\sigma_{\alpha}(P_1), \tau_{\alpha}(P_1), \sigma_{\alpha}\tau_{\alpha}(P_1)\}. $$
Furthermore, $\sigma^2 \not\in \langle \sigma_{\alpha}, \tau_{\alpha}\rangle$. 
Similar to the proof of (2), the assertion follows by Theorem \ref{main}. 

(4). 
Let $\alpha^2+\alpha-1=0$ and let $\sigma, \tau\in {\rm Aut}(\Bbb P^1)$ be represented by 
$$ \left(\begin{array}{cc} 
1 & -1 \\
1 & -\alpha
\end{array} \right), \ 
\left(\begin{array}{cc} 
0 & 1 \\
\alpha-1 & 1 
\end{array} \right) $$
respectively. 
Let $G_1=\langle \sigma \rangle$, $G_2=\langle \tau \rangle$, $P_1=(\alpha:2\alpha-1)$ and $P_2=(1:1+\alpha)$.
If $p \ne 2$, then $P_1 \ne P_2$. 
Note that
$$\{\sigma^i(P_2)| 1 \le i \le 4 \}=\{(1:0), (0:1), (1:1), (1:\alpha)\}=\{\tau^i(P_1)|1 \le i \le 4\}. $$
Condition (c) in Theorem \ref{main} is satisfied.    
Furthermore, $\sigma^5=1$ and $\tau^5=1$.
Conditions (a) and (b) are obviously satisfied in this case. 
\end{proof}

\begin{remark} 
For $\sigma \in {\rm Aut}(\Bbb P^1)$ as in the proof of Theorem \ref{main_rational}(1), we can show that the rational function $\sum_{i=0}^3\sigma^*(t) \in k(\Bbb P^1)=k(t)$ is a generator of $k(\Bbb P^1/\langle \sigma \rangle)$. 
Then, the birational embedding $\varphi: \Bbb P^1 \rightarrow \Bbb P^2$ is represented by 
$$ (2(t^4-6t^2+1)(2t-1):(4t^4-12t^2+8t-1)(t+1):2t(t+1)(t-1)(2t-1)). $$
Similarly, for Theorem \ref{main_rational}(2), we have a birational embedding
$$ ((t^2-\alpha)^2(t-\alpha'):(t^2-\alpha')^2(t-\alpha): t(t-1)(t-\alpha)(t-\alpha')).$$
\end{remark}

\begin{remark}
According to \cite[Theorem 1]{miura}, if $p=0$, $C=\Bbb P^1$ and $\deg \varphi(C)=6$, then the number of inner Galois points is bounded by two. 
Our curve in Theorem \ref{main_rational}(4) attains this bound.  
\end{remark}

Next, we consider elliptic curves. 
Let $p \ne 3$. 
Note that if an elliptic curve $E$ admits a Galois covering over $\Bbb P^1$ of degree three, then $E$ has an embedding $\varphi: E \rightarrow \Bbb P^2$ such that the image is the Fermat cubic. 
To consider the case where $\deg{\varphi(E)}=4$, we assume that $E \subset \Bbb P^2$ is the curve defined by $X^3+Y^3+Z^3=0$.

\begin{proof}[Proof of Theorem \ref{main_elliptic}] 
Let $\sigma$ be the automorphism of $E$ given by $(X:Y:Z) \mapsto (\omega X:Y:Z)$, where $\omega^2+\omega+1=0$. 
Then, $\sigma$ is of order three and $E/\langle \sigma \rangle \cong \Bbb P^1$. 
We take a point $Q \in E \setminus \{YZ=0\}$ such that $\sigma(Q) \ne Q$ and $\sigma^2(Q) \ne Q$. 
Note that there exists an involution $\eta$ such that $\eta(Q)=\sigma(Q)$, by  the linear system $|Q+\sigma(Q)|$. 
We take $\tau:=\eta \sigma^2 \eta$. 
Then, $\tau(Q)=\sigma(Q)$, $\tau$ is of order three and $E/\langle \tau \rangle \cong \Bbb P^1$. 
Let $G_1=\langle \sigma \rangle$ and $G_2=\langle \tau \rangle$. 
Then, condition (a) in Theorem \ref{main} is satisfied for $G_1$ and $G_2$. 
Furthermore, we take $P_1=\tau^2(Q)$ and $P_2=\sigma^2(Q)$. 
To prove (b) and (c) in Theorem \ref{main}, we have to only show that $\sigma^2(Q) \ne \tau^2(Q)$. 

Let $R_1=(1:0:0)$, $R_2=(0:1:0)$ and $R_3=(0:0:1)$. 
Assume by contradiction that $\tau^2(Q)=\sigma^2(Q)$. 
Then, $\eta(\sigma^2(Q))=\sigma^2(Q)$. 
For the divisor $D=Q+\sigma(Q)+\sigma^2(Q)$, $\eta^*(D)=D$.  
Since $D$ is given by $E \cap \overline{R_1Q}$ and the linear system $|D|$ is complete, where $\overline{R_1Q}$ is the line passing through $R_1$ and $Q$, $\eta$ is the restriction of some linear transformation $\hat{\eta}: \Bbb P^2 \cong \Bbb  P^2$. 
It is well known that the set of outer Galois points for $E$ is equal to $\Delta=\{R_1, R_2, R_3\}$ (see, for example, \cite{yoshihara2}). 
Since $\hat{\eta}(\Delta)=\Delta$ and $\hat{\eta}(\overline{R_1Q})=\overline{R_1Q} \not\ni R_2, R_3$, it can be noted that $\hat{\eta}(R_1)=R_1$ and $\hat{\eta}(\overline{R_2R_3})=\overline{R_2R_3}$. 
Then, $\hat{\eta}$ fixes points $R_1$, $Q$ and the point given by $\overline{R_1Q} \cap \overline{R_2R_3}$. 
This implies that $\hat{\eta}$ is identity on the line $\overline{R_1Q}$. 
This is a contradiction. 

By Theorem \ref{main}, the assertion follows. 
\end{proof}

\begin{remark}
According to \cite[Theorem 1]{miura}, if $p=0$ and $\deg \varphi(C)=4$, then the number of inner Galois points is bounded by four. 
When $C$ is an elliptic curve, it is not known whether or not the bound is sharp.  
\end{remark}

\begin{center} {\bf Acknowledgements} \end{center} 
The author is grateful to Professor Takeshi Takahashi, Professor Tomohide Terasoma and Professor Hisao Yoshihara for helpful conversations.

\end{document}